\documentclass{amsart}
\usepackage{amsmath}
\usepackage{graphicx}

\graphicspath{{figures/}}

\numberwithin{equation}{section}

\theoremstyle{definition}
\newtheorem{Definition}{Definition}[section]
\newtheorem{Conj}[Definition]{Conjecture}
\newtheorem{Cor}[Definition]{Corollary}
\newtheorem{Lem}[Definition]{Lemma}
\newtheorem{Note}[Definition]{Note}
\newtheorem{Prop}[Definition]{Proposition}

\newtheorem{Thm}[Definition]{Theorem}

\newcommand{\arccot}{\operatorname*{arccot}}
\newcommand{\Arg}{\operatorname*{Arg}}
\newcommand{\colonequal}{\mathrel{\mathop:}=}
\newcommand{\lcm}{\operatorname*{lcm}}
\newcommand{\nequiv}{\mathrel{\not\equiv}}

\begin{document}

\title{The iterated integrals of $\ln(1 + x^2)$}

\author[T. Amdeberhan]{Tewodros Amdeberhan}
\address{Department of Mathematics,
Tulane University, New Orleans, LA 70118}
\email{tamdeber@math.tulane.edu}

\author[C. Koutschan]{Christoph Koutschan}
\address{Research Institute for Symbolic Computation,
Johannes Kepler University, 4040 Linz, Austria}
\email{ckoutsch@risc-uni-linz.ac.at}

\author[V. Moll]{Victor H. Moll}
\address{Department of Mathematics,
Tulane University, New Orleans, LA 70118}
\email{vhm@math.tulane.edu}

\author[E. Rowland]{Eric S. Rowland}
\address{Department of Mathematics,
Tulane University, New Orleans, LA 70118}
\email{erowland@tulane.edu}

\date{April 4, 2011}

\begin{abstract}
For a polynomial $P$, we consider the sequence of iterated integrals of
$\ln P(x)$. This sequence
is expressed in terms of the zeros of $P(x)$. In the special case of 
$\ln(1+x^{2})$, arithmetic properties of certain 
coefficients arising are described.  Similar 
observations are made for $\ln (1 + x^3)$.
\end{abstract}

\maketitle

\section{Introduction} \label{sec-intro}
\setcounter{equation}{0}

The evaluation of integrals, a subject that had an important role in 
the $19^{th}$ century, has been given a new life with the development of 
symbolic mathematics software such as \emph{Mathematica} or \emph{Maple}. 
The question of indefinite integrals was provided with an algorithmic approach 
beginning with work of J. Liouville~\cite{liouville-1834c} discussed in 
detail in Chapter IX of Lutzen~\cite{lutzen}. A more modern treatment can be found in Ritt~\cite{ritt-1948}, R. H. Risch~\cite{risch1,risch2}, and M. Bronstein~\cite{bronstein2}. 

The absence of a complete algorithmic solution to the problem of evaluation
of definite integrals justifies the validity of tables of integrals such as~\cite{apelblat, gr, prudnikov1}.
These collections have not been superseded, yet, by the software mentioned above. 

The point of view illustrated in this paper is that the quest for evaluation 
of definite integrals may take the reader to unexpected parts of mathematics. 
This has been described by one of the authors 
in~\cite{moll-notices,moll-seized}. The goal here is to consider the sequence 
of {\em iterated integrals} of a function $f_{0}(x)$, defined by
\begin{equation}
	f_{n}(x) = \int_{0}^{x} f_{n-1}(t) \, dt \quad \text{if $n \geq 1$}.
\end{equation}
This formula carries the implicit normalization $f_{n}(0) = 0$ for $n \geq 1$.

A classical formula for the iterated integrals is given by 
\begin{equation}
\label{class-1}
f_{n}(x) = \frac{d^{-n}}{dx^{-n}} f(x) = 
\frac{1}{(n-1)!} \int_{0}^{x} f_{0}(t) \, 
(x-t)^{n-1} \, dt. 
\end{equation}
Expanding the kernel $(x-t)^{n-1}$ gives $f_{n}$ in terms of the moments
\begin{equation}
	M_{j}(x) = \int_{0}^{x} t^{j} f_{0}(t) \, dt
\end{equation}
as
\begin{equation}
	f_{n}(x) = \sum_{j=0}^{n-1} (-1)^{j} \frac{x^{n-1-j}}{j! \, (n-1-j)!} M_{j}(x). \label{expan-1}
\end{equation}

The work presented here deals with the sequence starting at 
$f_{0}(x) = \ln(1 + x^{N})$. The main observation is that the 
closed-form expression of the iterated integrals contains a pure polynomial 
term and a linear combination of transcendental functions with polynomial 
coefficients. Some arithmetical properties of the 
pure polynomial term are described.

\section{The iterated integral of $\ln(1+x)$} \label{sec-case1}
\setcounter{equation}{0}

The iterated integral of $f_{0}(x) = \ln(1+x)$ was described in~\cite{iter-loga1}. This sequence has the form
\begin{equation}
	f_{n}(x) = A_{n,1}(x) + B_{n,1}(x) \ln(1+x) 
\label{loga-0}
\end{equation}
where 
\begin{align}
	A_{n,1}(x) &= - \frac{1}{n!} \sum_{k=1}^{n} \binom{n}{k} \left( H_{n} - H_{n-k} \right) x^{k} = -\frac{1}{n!} \sum_{k=1}^n \frac{x^k (x+1)^{n-k}}{k}, \label{An-def} \\
	B_{n,1}(x) &= \frac{1}{n!} (1+x)^{n}, \nonumber
\end{align}
where $H_{n} = 1 + \frac{1}{2} + \dots + \frac{1}{n}$ is the $n$th harmonic number. 

The expression for $B_{n,1}(x)$ is easily guessed from the 
symbolic computation of the 
first few values. The corresponding closed form 
for $A_{n,1}(x)$ was more difficult to find experimentally. Its study 
began with the analysis of its 
denominators, denoted here by $\alpha_{n,1}$. The fact that the ratio
\begin{equation}
	\beta_{n,1} \colonequal \frac{\alpha_{n,1}}{n \, \alpha_{n-1,1}} \label{beta-1}
\end{equation}
satisfies 
\begin{equation}
\beta_{n,1} = \begin{cases} 
 p & \text{if $n$ is a power of the prime $p$} \\
 1 & \text{otherwise}
 \end{cases}
\end{equation}
was the critical observation in obtaining the closed form $A_{n,1}(x)$ 
given in~\eqref{An-def}. We recognize $\beta_{n,1}$ as $e^{\Lambda(n)}$, where 
\begin{equation}
\label{beta-lambda}
\Lambda(n) = \begin{cases} 
 \ln p & \text{if $n$ is a power of the prime $p$} \\
 0 & \text{otherwise}
 \end{cases}
\end{equation}
is the von Mangoldt function.  This yields 
\[
	\alpha_{n,1} = n! \prod_{j=2}^{n} \beta_{j,1} = n! \prod_{j=2}^{n} e^{\Lambda(j)},
\]
and the relation 
\begin{equation}
e^{\Lambda(n)} = \frac{\lcm(1, \dots, n)}{\lcm(1, \dots, n-1)}
\end{equation}
shows that 
\begin{equation}
\alpha_{n,1} = n! \, \lcm(1, \dots, n).
\label{formula-alpha1}
\end{equation}

\begin{Note}
The harmonic number $H_{n}$ appearing in (\ref{An-def})  has 
challenging arithmetical properties. Written in reduced form as 
\begin{equation}
	H_{n} = \frac{N_{n}}{D_{n}},
\end{equation}
the denominator $D_{n}$ divides the least common multiple $L_{n} \colonequal \lcm(1, 2, \dots, n)$.  The complexity of the ratio $L_{n}/D_{n}$ can be seen in Figure~\ref{ratio1}.  It has been conjectured~\cite[page 304]{graham1} that $D_{n} = L_{n}$ for infinitely many values of $n$.
\end{Note}

\begin{figure}
\begin{center}
	\includegraphics{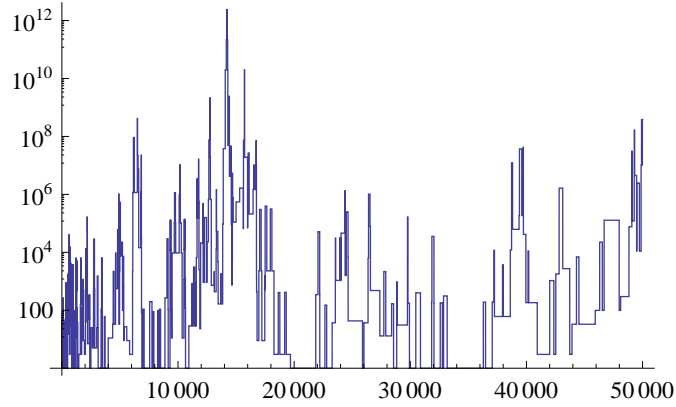}
	\caption{Logarithmic plot of the ratio $L_{n}/D_{n}$.}\label{ratio1}
\end{center}
\end{figure}

The expressions for $A_{n,1}(x)$ and $B_{n,1}(x)$ can also be derived from~\eqref{expan-1}.  Letting $f_{0}(x) = \ln(1+x)$ yields
\begin{equation}
f_{n}(x) = \sum_{j=0}^{n-1} (-1)^{j} \frac{x^{n-1-j}}{j! (n-1-j)!} 
\int_{0}^{x} t^{j} \ln(1+t) \, dt. 
\label{form-1a}
\end{equation}
Integration by parts gives 
\begin{equation}
\int_{0}^{x} t^{j} \ln(1+t) \, dt = 
\frac{x^{j+1} \, \ln(1+x)}{j+1} - \frac{1}{j+1} \int_{0}^{x} \frac{t^{j+1} 
\, dt}{1+t}.
\end{equation}
Replacing in~\eqref{form-1a} shows that the contribution of the first term 
reduces simply to $x^{n} \ln(1+x)$. Therefore
\begin{equation}
	f_{n}(x) = \frac{1}{n!} x^{n} \ln(1+x) + \frac{1}{n!} \sum_{j=1}^{n} (-1)^{j} \binom{n}{j} x^{n-j} \int_{0}^{x} \frac{t^{j} \, dt}{1+t}. 
\label{form-1b}
\end{equation}

It remains to provide a closed form for the integrals 
\begin{equation}
	I_{j} \colonequal \int_{0}^{x} \frac{t^{j}}{1+t} \, dt.
\end{equation}
These can be produced by elementary methods by writing 
\begin{equation}
	\frac{t^{j}}{1+t} = \frac{t^{j} - (-1)^{j}}{1+t} + \frac{(-1)^{j}}{1+t}.
\end{equation}
Replacing in~\eqref{form-1b} gives
\begin{eqnarray*}
	f_{n}(x)& = & \frac{1}{n!} x^{n} \ln(1+x) \\
		& + & \frac{1}{n!} \sum_{j=1}^{n} (-1)^{j} \binom{n}{j} x^{n-j} \int_{0}^{x} \frac{t^{j} - (-1)^{j}}{t+1} \, dt \\
		& + & \frac{1}{n!} \sum_{j=1}^{n} \binom{n}{j} x^{n-j} \int_{0}^{x} \frac{dt}{1+t}.
\end{eqnarray*}
The first and last line add up to $(x+1)^{n} \ln(1+x)/n!$, which yields the 
closed-form expression for $B_{n,1}(x)$. Expanding the quotient 
in the second line produces
\begin{equation}
	\frac{1}{n!} \sum_{j=1}^{n} (-1)^{j} \binom{n}{j} x^{n-j} \sum_{r=0}^{j-1} \frac{(-1)^{r}}{j-r} x^{j-r} = 
	\frac{1}{n!} \sum_{j=0}^{n-1} \binom{n}{j} x^{j} \sum_{r=1}^{n-j} \frac{(-1)^{r}}{r} x^{r}.
\end{equation} 
The double sum can be written as 
\begin{equation}
	\frac{1}{n!} \sum_{j=0}^{n} \sum_{r=1}^{n-j} \binom{n}{j} \frac{(-1)^r}{r} x^{j+r} =
	\frac{1}{n!} \sum_{a=1}^{n} \left[ \sum_{r=1}^{a} \binom{n}{a-r} \frac{(-1)^{r}}{r} \right] x^{a}.
\end{equation}
The expression for $A_{n,1}(x)$ now follows from the identity
\begin{equation}
	\sum_{r=1}^{a} \binom{n}{a-r} \frac{(-1)^{r}}{r} = - \binom{n}{a} \left[ H_{n} - H_{n-a} \right].
\end{equation}
An equivalent form, with $m = n-a$, is given by
\begin{equation}
	U(a) \colonequal \sum_{r=1}^{a} \frac{(-1)^{r-1} \binom{a}{r}}{r \, \binom{m+r}{r}} = H_{m+a} - H_{m}.
\end{equation}
To establish this identity, we employ the WZ method~\cite{aequalsb}. Define the pair of functions
\begin{equation}
	F(r,a) = \frac{(-1)^{r-1} \binom{a}{r}}{r \, \binom{m+r}{r}} \quad \text{and} \quad G(r,a) = \frac{(-1)^{r} \binom{a}{r-1}}{(m+a+1) \, \binom{m+r-1}{r-1}}.
\end{equation}
It can be easily checked that 
\begin{equation}
	F(r,a+1) - F(r,a) = G(r+1,a) - G(r,a).
\end{equation}
Summing both sides of this equation over $r$, from $1$ to $a+1$, leads to 
\begin{equation}
	U(a+1) - U(a) = \frac{1}{m+a+1}.
\end{equation}
Now sum this identity over $a$, from $1$ to $k-1$, to obtain
\begin{equation}
	U(k) - U(1) = \sum_{a=1}^{k-1} \frac{1}{m+a+1} = \sum_{a=m+2}^{m+k} \frac{1}{r} = H_{m+k} - H_{m+1}.
\end{equation}
Combining this with the initial condition $U(1) = \tfrac{1}{m+1}$ gives the result.

\section{The method of roots}\label{sec-roots}
\setcounter{equation}{0}

The iterated integrals of the 
function $f_{0}(x) = \ln P(x)$ for a general polynomial
\begin{equation}
	P(x) = \prod_{j=1}^{m} (x + z_{j})
\end{equation}
are now expressed in terms of the roots $z_j$ using an explicit expression for 
the iterated integrals of $f_{0}(x) = \ln(x+a)$.

\begin{Thm}
\label{thm-loga}
The iterated integral of $f_{0}(x) = \ln(x+a)$ is given by 
\begin{equation}
f_{n}(x) =
- \frac{1}{n!} \sum_{k=1}^{n} \frac{x^{k} (x+a)^{n-k}}{k}
- \frac{(x+a)^{n} - x^{n}}{n!} \, \ln a
+ \frac{(x+a)^{n}}{n!} \ln(x+a).
\label{form-gen-a}
\end{equation}
\end{Thm}

\begin{proof}
A symbolic calculation of the first few values suggests the ansatz 
$f_{n}(x) = S_{n}(x) + T_{n}(x) \ln(x+a)$ for some polynomials $S_{n}, T_{n}$. The relation $f_{n}' = f_{n-1}$ and the form of $S_{n}, T_{n}$ 
given in~\eqref{form-gen-a} give the result by induction.
\end{proof}

In the special case $P(x) = 1 + x^N$, the previous result can be made more
explicit.

\begin{Thm}
\label{thm-arctan}
Let $a = u + i v$ be a root of $1+x^{N}=0$. Then the contribution of $a$ and $\bar{a} = u - i v$
to the iterated integral of $\ln(1+x^{N})$ is given by 
\begin{multline*}
	- \frac{1}{n!} \sum_{k=1}^{n} \frac{x^{k}}{k} \left[ (x+a)^{n-k} + (x+ \bar{a})^{n-k} \right] \\
	+ \frac{1}{i n!} \left[ (x+a)^{n} - (x + \bar{a})^{n} \right] \arctan \left( \frac{vx}{1+ux} \right) \\
	+ \frac{(x+a)^{n} + (x+\bar{a})^{n}}{2n!} \ln [ (1+ux)^{2} + v^{2}x^{2} ].
\end{multline*}
\end{Thm}

\begin{proof}
First observe that 
$\ln(x+a) - \ln a = \ln(\bar{a}x+1)$; hence for $f_{0}(x) = \ln(x+a)$ Theorem~\ref{thm-loga} takes the form
\begin{equation}
	f_n(x) = -\frac{1}{n!} \sum_{k=1}^{n} \frac{x^{k} (x+a)^{n-k}}{k} + \frac{x^{n}}{n!} \ln a + \frac{(x+a)^{n}}{n!} \ln( \bar{a}x + 1).
\end{equation}
Since
\begin{align*}
	\ln(ax + 1) &= \ln |ax + 1| + i \Arg(ax+1), \\
	\ln(\bar{a}x + 1) &= \ln |ax + 1| - i \Arg(ax+1),
\end{align*}
and $\ln a + \ln \bar{a} = 2 \ln |a| = 0$, it follows that the total contribution of $a$ and $\bar{a}$ is given by 
\begin{multline*}
-\frac{1}{n!} \sum_{k=1}^{n} \frac{x^{k}}{k}  \left[ (x+a)^{n-k} + (x+ \bar{a})^{n-k} \right] + \frac{\left[ (x+a)^{n} \ln( \bar{a}x+1) + (x+ \bar{a})^{n} \ln(ax+1) \right]}{n!} \\
	= -\frac{1}{n!} \sum_{k=1}^{n} \frac{x^{k}}{k} \left[ (x+a)^{n-k} + (x + \bar{a})^{n-k} \right] + \frac{ \left[ (x+a)^{n} + (x + \bar{a})^{n} \right]}{n!} \ln |a x + 1 | \\
	- \frac{ i \left[ (x+a)^{n} - (x + \bar{a})^{n} \right]}{n!} \Arg(ax+1).
\end{multline*}
The stated result comes from expressing the logarithmic terms in their 
real and imaginary parts.
\end{proof}

\begin{Cor}
\label{cor-1}
Let $n \in \mathbb{N}$. Then 
\[
\sum_{k=1}^{n} \frac{1}{k} \int_{0}^{x} t^{k} (t+a)^{n-k} \, dt = 
\frac{1}{n+1} \sum_{k=1}^{n} \frac{x^{k} (x+a)^{n+1-k}}{k} + 
\frac{ \left[ x^{n+1} - (x+a)^{n+1} + a^{n+1} \right]}{(n+1)^{2}}.
\]
\end{Cor}
\begin{proof}
Integrate both sides of the identity in Theorem \ref{thm-loga} 
and use the relation $f_{n-1}' = f_{n}$ to obtain the result inductively. 
\end{proof}

\begin{Note}

The identity in Corollary \ref{cor-1} can be expressed in terms of the 
function
\begin{equation}
	\Phi_{n}(x,a):= \sum_{k=1}^{n} \frac{1}{k} x^{k}(x+a)^{n-k}
\end{equation}
in the form 
\begin{equation}
	\int_{0}^{x} \Phi_{n}(t,a) \, dt = \frac{x+a}{n+1} \Phi_{n}(x,a) + \frac{1}{(n+1)^{2}} \left[ x^{n+1}+a^{n+1}-(x+a)^{n+1} \right].
\end{equation}
The function $\Phi_{n}(x,a)$ admits the hypergeometric representation 
\[
	\Phi_{n}(x,a) = - \frac{x^{n+1}}{(n+1)(x+a)} \,_{2}F_{1}\Big(\begin{array}{c}
1, 1+n\\
2+n \end{array} ; \frac{x}{x+a}\Big)
-(x+a)^{n} \ln \left( \frac{a}{x+a} \right).
\]

With this representation, the identity in Corollary \ref{cor-1} now becomes

\[
	\int_{0}^{x} \left( \frac{t}{1-t} \right)^{n+1} 
\,_{2}F_{1}\Big(\begin{array}{c}
1, 1+n\\
2+n \end{array} ; t\Big) \, \frac{dt}{1-t} = 
\frac{1}{n+1} \left( \frac{x}{1-x} \right)^{n+1} 
\left[ 
\,_{2}F_{1}\Big(\begin{array}{c}
1, 1+n\\
2+n \end{array} ; x \Big) - 1 \right].
\]
\end{Note}

\section{The iterated integral of $\ln(1+x^{2})$} \label{sec-case2}
\setcounter{equation}{0}

In this section we consider the iterated integral of $f_{0}(x) = \ln(1+x^{2})$ defined by
\begin{equation}\label{loga-2}
	f_{n}(x) = \int_{0}^{x} f_{n-1}(t) \, dt.
\end{equation}

The first few examples, given by 
\begin{align*}
	f_{1}(x) &= -2x + 2 \arctan x + x \, \ln(1+ x^{2}) \\
	f_{2}(x) &= - \tfrac{3}{2}x^{2} + 2x \, \arctan x + \tfrac{1}{2}(x^{2}-1) \ln(1+x^{2}) \\
	f_{3}(x) &= - \tfrac{11}{18}x^{3} + \tfrac{1}{3}x + (x^{2} - \tfrac{1}{3}) \, \arctan x + \left( \tfrac{1}{6}x^{3} - \tfrac{1}{2}x \right) \ln(1+x^{2}),
\end{align*}

suggest the form 
\begin{equation}
f_{n}(x) = A_{n,2}(x) + B_{n,2}(x) \arctan x + C_{n,2}(x) \ln(1+x^{2})
\label{form}
\end{equation}
for some polynomials $A_{n,2}, B_{n,2}, C_{n,2}$.
Theorem~\ref{thm-arctan} can be employed to obtain a closed form for these 
polynomials.  It follows that $f_n(x)$ satisfies
\begin{multline}
\label{expression}
	n! f_{n}(x) = - \sum_{k=1}^{n} \frac{x^{k}}{k} \left[ (x+i)^{n-k} + (x-i)^{n-k} \right] \\
	- i \left[ (x+i)^{n} - (x-i)^{n} \right] \arctan x + \frac{1}{2} \left[ (x+i)^{n} + (x-i)^{n} \right] \ln(1 + x^{2}).
\end{multline}

The expressions for $A_{n,2}, \, B_{n,2}, \, C_{n,2}$ may be read from here.

\subsection{Recurrences} \label{sec-recu}

The polynomials $A_{n,2}, B_{n,2}, C_{n,2}$ can also be found as solutions to 
certain recurrences.
Differentiation of~\eqref{loga-2} yields $f_{n}'(x) = f_{n-1}(x)$. It is 
easy to check that this relation, with the initial conditions 
$f_{n}(0) = 0$ and $f_{0}(x) = \ln(1+x^{2})$, is equivalent to~\eqref{loga-2}.
Replacing the ansatz~\eqref{form} produces
\begin{multline*}
A_{n,2}'(x) + B_{n,2}'(x) \arctan x + \frac{B_{n,2}(x)}{1+x^{2}} + 
C_{n,2}'(x) \ln(1+x^{2}) + C_{n,2}(x) \frac{2x}{1+x^{2}} \\
= A_{n-1,2}(x) + B_{n-1,2}(x) \arctan x + C_{n-1,2}(x) \ln(1+x^{2}).
\end{multline*}
A natural linear independence assumption yields the system of recurrences
\begin{align}
	B_{n,2}'(x) &= B_{n-1,2}(x) \label{recB} \\
	B_{0,2}(x) &= 0 \nonumber \\
	C_{n,2}'(x) &= C_{n-1,2}(x) \label{recC} \\
	C_{0,2}(x) &= 1 \nonumber \\
	A_{n,2}'(x) &= A_{n-1,2}(x) - \frac{B_{n,2}(x) + 2xC_{n,2}(x)}{1+x^{2}} \label{recA} \\
	A_{0,2}(x) &= 0.  \nonumber
\end{align}

\begin{Note}
The definition~\eqref{loga-2} determines completely the function $f_{n}(x)$.
In particular, given the form~\eqref{form}, the polynomials $A_{n,2}, 
B_{n,2}$ and $C_{n,2}$ are uniquely specified. Observe however that 
the recurrence~\eqref{recB} does not determine $B_{n,2}(x)$
uniquely. At each step, there is a constant of integration to be determined. 
In order to address this ambiguity, the first few values of $B_{n,2}(0)$ are 
determined {\em empirically}, and the condition 
\begin{equation}
\label{initB}
B_{n,2}(0) = \begin{cases} 
 2 (-1)^{\tfrac{n-1}{2}}/n! & \text{if $n$ is odd} \\
 0 & \text{if $n$ is even}
 \end{cases}
\end{equation}
is added to the recurrence~\eqref{recB}. The polynomials $B_{n,2}(x)$ are now 
uniquely determined.  Similarly, the initial condition
\begin{equation}
C_{n,2}(0) = \begin{cases} 
 (-1)^{\tfrac{n}{2}}/n! & \text{if $n$ is even} \\
 0 & \text{if $n$ is odd}
 \end{cases}
\label{initC}
\end{equation}
adjoined to~\eqref{recC}, determines $C_{n,2}$. The initial condition 
imposed on $A_{n,2}$ is simply $A_{n,2}(0) = 0$. 
\end{Note}

The recurrence~\eqref{recB} is then employed to produce a list of the first 
few values of $B_{n,2}(x)$. These are then used to guess the closed-form
expression for this family.  The same is true for $C_{n,2}(x)$. 

\begin{Prop}
\label{prop-easy}
The recurrence~\eqref{recB} and the (heuristic) initial condition 
\eqref{initB} yield
\begin{align}
	B_{n,2}(x)	&= \frac{2}{n!} \sum_{j=0}^{\tfrac{n-1}{2}} (-1)^{j} \binom{n}{2j+1} x^{n-2j-1} \label{form-B} \\
			&= \frac{1}{i \, n!} \left[ (x+i)^{n}- (x-i)^{n} \right]. \nonumber
\end{align}
Similarly, the polynomial $C_{n,2}$ is given by 
\begin{align}
	C_{n,2}(x)	& = \frac{1}{n!} 
 \sum_{j=0}^{\left\lfloor \tfrac{n}{2} \right\rfloor} (-1)^{j} \binom{n}{2j} x^{n-2j} \label{form-C} \\
			& = \frac{1}{2 \, n!} \left[ (x+i)^{n}+ (x-i)^{n} \right]. \nonumber
\end{align}
In particular, the degree of $B_{n,2}$ is $n-1$, and the degree of $C_{n,2}$ is $n$.
\end{Prop}

\begin{proof}
This follows directly from the recurrences~\eqref{recB} and \eqref{recC}.
\end{proof}

\begin{Cor}
The recurrence for $A_{n,2}$ can be written as 
\begin{equation}
A_{n,2}'(x) = A_{n-1,2}(x) - \frac{1}{n!} \left[ (x+i)^{n-1} + (x-i)^{n-1} 
\right]. 
\end{equation}
In particular, the degree of $A_{n,2}$ is $n$.
\end{Cor}

\begin{proof}
Simply replace the explicit expressions for $B_{n,2}$ and 
$C_{n,2}$ in the recurrence~\eqref{recA}.
\end{proof}

\subsection{Trigonometric forms}

A trigonometric form of the polynomials 
$B_{n,2}$ and $C_{n,2}$ is established next.

\begin{Prop}
The polynomials $B_{n,2}$ and $C_{n,2}$ are given by 
\begin{align*}
	B_{n,2}(x) &= \frac{2}{n!} (x^{2}+1)^{n/2} \sin( n \arccot x) \\
	C_{n,2}(x) &= \frac{1}{n!} (x^{2}+1)^{n/2} \cos( n \arccot x).
\end{align*} In particular, 
\begin{equation}
\frac{C_{n,2}(x)}{B_{n,2}(x)} = \frac{1}{2} \cot( n \arccot x). \label{strange}
\end{equation}
\end{Prop}

\begin{proof}
The polar form 
\begin{equation}
	x+i = \sqrt{x^{2}+1} \, \left[ \cos (\arccot x ) + i \sin (\arccot x) \right]
\end{equation}
produces
\begin{equation}
	(x+i)^{n} = (x^{2}+1)^{n/2} \left[ \cos (n \arccot x) + i \sin (n \arccot x) \right].
\end{equation}
A similar expression for $(x-i)^{n}$ gives the result.
\end{proof}

\begin{proof}
A second proof follows from the Taylor series 
\begin{equation}
\frac{\sin( z \arctan t ) }{(1+t^{2})^{z/2}} = 
\sum_{k=0}^{\infty} \frac{(-1)^{k} (z)_{2k+1} }{(2k+1)!} t^{2k+1} 
\label{series-sine}
\end{equation}
and 
\begin{equation}
\frac{\cos( z \arctan t ) }{(1+t^{2})^{z/2}} = 
\sum_{k=0}^{\infty} \frac{(-1)^{k} (z)_{2k} }{(2k)!} t^{2k} 
\end{equation}
where $(z)_n$ denotes the Pochhammer symbol. These series were
established in~\cite{boesmo} in the context of integrals related to the 
Hurwitz zeta function. 

Indeed, the formula for $B_{n,2}(x)$ comes from \eqref{series-sine} replacing
 $t$ by $1/x$ and $z$ by $-n$ to obtain
\begin{equation}
\sin(n \arccot x) \, (x^{2}+1)^{n/2} = - x^{n} 
\sum_{k=0}^{\infty} \frac{(-1)^{k} (-n)_{2k+1}}{(2k+1)!} x^{-2k-1}.
\end{equation}
The result~\eqref{form-B} now follows from the identity
\begin{equation}
(-n)_{2k+1} = \begin{cases} 
	- n!/(n-2k-1)! & \text{if $2k+1 \leq n$} \\
	0 & \text{otherwise.}
\end{cases}
\end{equation}
A similar argument gives the form of $C_{n,2}(x)$ in~\eqref{form-C}.
\end{proof}

\begin{Note}
The rational function $R_{n}$ that gives 
\begin{equation}
\cot(n \theta) = R_{n}(\cot \theta) 
\end{equation}
appears in~\eqref{strange} in the form 
\begin{equation}
R_{n}(x) = \frac{2C_{n,2}(x)}{B_{n,2}(x)}. 
\end{equation}
This rational function plays a crucial role in the development of 
{\em rational Landen transformations}~\cite{manna-moll2}.  These are 
transformations of the coefficients of a rational integrand that preserve 
the value of a definite integral. For example, the map 
\begin{eqnarray*}
	a & \mapsto & a \left( (a+3c)^{2} - 3b^{2} \right)/\Delta \\
	b & \mapsto & b \left( 3(a-c)^{2} - b^{2} \right)/\Delta \\
	c & \mapsto & c \left( (3a+c)^{2} - 3b^{2} \right)/\Delta,
\end{eqnarray*}
where $\Delta = (3a+c)(a+3c) - b^{2}$, preserves the value of
\begin{equation}
\int_{-\infty}^{\infty} \frac{dx}{ax^{2}+bx+c} = \frac{2 \pi}{\sqrt{4ac-b^{2}}}.
\end{equation}
The reader will find in~\cite{manna-moll3} a survey of this type of 
transformation and~\cite{manna-moll1} the example given above.  The
reason for the appearance of $R_{n}(x)$ in 
the current context remains to be clarified.
\end{Note}

\subsection{An automatic derivation of a recurrence for $A_{n,2}$}\label{sec-automatic}

The formula~\eqref{class-1} for 
for the iterated integral can be used in the context of computer algebra 
methods. In the case discussed here, the integral
\begin{equation}\label{int2}
 I_n(x) = \frac{1}{(n-1)!} \int_0^x (x-t)^{n-1} \ln(1+t^2) \, dt
\end{equation}
gives the desired iterated integrals of $\ln(1+x^2)$ for $n\geq 1$.

A standard application of the holonomic systems approach, as implemented
in the \emph{Mathematica} package 
\texttt{HolonomicFunctions}~\cite{Koutschan09}, yields a 
recurrence in~$n$ for~\eqref{int2}. The reader will find 
in~\cite{moll-gr18} a description of the use of this package in the 
evaluation of definite integrals.  The recurrence 
\begin{multline}
	n^2(n-1)I_n(x) \\
	= x(3n-2)(n-1) \, I_{n-1}(x) - \left(3nx^2-4x^2+n\right)I_{n-2}(x) \\ + x \left(x^2+1\right) I_{n-3}(x)  \label{rec2}
\end{multline}
is delivered immediately by the package. Using the linear independence
of $\arctan x$ and $\ln(1+x^{2})$, it follows that each of the sequences
$A_{n,2}$, $B_{n,2}$, and $C_{n,2}$ must also 
satisfy the recurrence~\eqref{rec2}.
Symbolic methods for solving recurrences are employed next to
produce the explicit expressions for $A_{n,2}$, $B_{n,2}$, and $C_{n,2}$ given 
above. 

Petkov\v{s}ek's algorithm Hyper~\cite{Petkovsek92} 
(as implemented in the \emph{Mathematica} package \texttt{Hyper}, for example)
computes a basis of hypergeometric solutions of a linear recurrence with
polynomial coefficients. Given~\eqref{rec2} as input, it outputs the two
solutions $(x+i)^n/n!$ and $(x-i)^n/n!$. The initial values are used to 
obtain the correct linear combinations of these solutions. This produces the 
expressions for $B_{n,2}(x)$ and $C_{n,2}(x)$ given in Proposition~\ref{prop-easy}.

However, the third solution is not hypergeometric and it will give the
polynomials~$A_{n,2}(x)$. It can be found by Schneider's \emph{Mathematica}
package \texttt{Sigma}~\cite{Schneider07}:
\[
	A_{n,2}(x)=\frac{i}{n!} \left(x\left((x+i)^n-(x-i)^n\right)
	+ \sum_{k=2}^n \frac{x^k \left((x-i)^{n-k+1} - (x+i)^{n-k+1}\right)}{(k-1) k} \right),
\]
with the initial values
\[
	A_{0,2}(x) = 0,\qquad A_{1,2}(x) = -2x,\qquad A_{2,2}(x) = -\tfrac{3}{2}x^2.
\]

In summary: 

\begin{Thm}
\label{thm-an2}
Define $a_{k} = k(k-1)$ for $k \geq 2$ and $a_{1}=-1$. The polynomial $A_{n,2}(x)$
introduced in~\eqref{form} is given by
\begin{equation}
\label{an2-type1}
A_{n,2}(x) = \frac{1}{i \, n!} \sum_{k=1}^{n} \frac{x^{k}}{a_{k}} 
\left[ (x+i)^{n-k+1}-(x-i)^{n-k+1} \right].
\end{equation}
This can be written as
\begin{equation}
A_{n,2}(x) = \frac{1}{n!} \sum_{k=1}^{n} \frac{(n-k+1)!}{a_{k}} 
x^{k} B_{n-k+1,2}(x).
\end{equation}
\end{Thm}

Note that the expression for $A_{n,2}$ given before is equivalent to the forms 
appearing in Theorem~\ref{thm-an2}.

\begin{Note} 
Similar procedures applied to the case of $\ln(1+x)$ 
yield the evaluation given in~\eqref{An-def}.
\end{Note}

\section{Arithmetical properties}\label{sec-arith}

In this section we discuss arithmetical properties of the polynomials $B_{n,2}$ and $A_{n,2}$.  The explicit formula for $B_{n,2}$ produces some elementary results. 

\begin{Prop}
Let $m, n \in \mathbb{N}$ such that $m$ divides $n$. Then $B_{m,2}(x)$ divides 
$B_{n,2}(x)$ as polynomials in $\mathbb{Q}[x]$.
\end{Prop}

\begin{proof}
This follows directly from~\eqref{form-B} and 
the divisibility of $a^{n}-b^{n}$ by $a^{m}-b^{m}$.
\end{proof}

For odd $n$, the quotient of $B_{2n,2}(x)$ by $B_{n,2}(x)$ admits a simple expression. 

\begin{Prop}
Let $n \in \mathbb{N}$. Define 
\begin{equation}
B^{*}_{n,2}(x) = x^{\deg B_{n,2}} B_{n,2}(1/x).
\end{equation}
Then, for $n$ odd, 
\begin{equation}
\binom{2n}{n} B_{2n,2}(x) = (-1)^{\tfrac{n-1}{2}}x B_{n,2}(x) B^{*}_{n,2}(x).
\end{equation}
In particular, the sequence of coefficients in $B_{2n}(x)$ is palindromic.
\end{Prop}

\begin{proof}
The proof is elementary. Observe that 
\begin{align*}
	B_{n,2}^{*}(x)	&= \frac{x^{n-1}}{in!} \left[ \left( \frac{1}{x} + i \right)^{n} - \left( \frac{1}{x} - i \right)^{n} \right] \\
			&= \frac{1}{ix n!} \left[ (1 + ix)^{n} - (1-ix)^{n} \right] \\
			&= \frac{i^{n-1}}{n! x} \left[ (x-i)^{n} - (-1)^{n}(x+i)^{n} \right].
\end{align*}
It follows that 
\begin{equation}
	B_{n,2}^{*}(x) = \frac{(-1)^{\tfrac{n-1}{2}}}{xn!} \left[ (x+i)^{n} + (x-i)^{n} \right],
\end{equation}
and the result now follows directly.
\end{proof}

The explicit expression~\eqref{an2-type1} for the polynomial $A_{n,2}$ can be
written in terms of the polynomials
\begin{equation}
	\varphi_{m}(x) = (x+i)^{m} - (x-i)^{m}
\end{equation}
as
\begin{equation}
\label{an2-type2}
	A_{n,2}(x) = \frac{i}{n!} \left[ x \varphi_{n}(x) - \sum_{k=2}^{n} \frac{x^{k} \varphi_{n-k+1}(x)}{k(k-1)} \right].
\end{equation}
The polynomial $A_{n,2}$ is of degree $n$ and has rational coefficients. 

By analogy with the properties of denominators of $A_{n,1}(x)$ mentioned in Section~\ref{sec-case1} and discussed at greater length in~\cite{iter-loga1},
we now study the denominators $A_{n,2}(x)$ from an arithmetic point of view.  The first result is elementary.

\begin{Prop}
Let 
\begin{equation}
	\alpha_{n,2} \colonequal \text{denominator of $A_{n,2}(x)$}.
\label{den-an2}
\end{equation}
Then $\alpha_{n,2}$ divides $n! \, \lcm(1, 2, \dots, n)$.
\end{Prop}

\begin{proof}
The result follows from~\eqref{an2-type2} and the fact that the 
polynomials $\varphi_{m}(x)$ have integer coefficients. 
\end{proof}

As in~\eqref{beta-1}, it is useful to consider the ratio
\begin{equation}
	\beta_{n,2} \colonequal \frac{\alpha_{n,2}}{n \, \alpha_{n-1,2}}.
\end{equation}
Symbolic computations suggest the following.

\begin{Conj}
\label{conj-1}
The sequence $\beta_{n,2}$ is given by 
\begin{equation}
\label{betan2}
\beta_{n,2} =
\begin{cases} 
	p		& \text{if $n = p^r$ for some prime $p$ and $r \in \mathbb{N}$ and $n \neq 2 \cdot 3^{m} + 1$} \\
	\frac{1}{3}	& \text{if $n = 2 \cdot 3^{m}$ for some $m \in \mathbb{N}$} \\
	3p		& \text{if $n = 2 \cdot 3^{m} + 1$ and $n = p^{r}$ for some $m, r \in \mathbb{N}$} \\
	3		& \text{if $n = 2 \cdot 3^{m} + 1$ for some $m \in \mathbb{N}$ and $n \neq p^r$} \\
	1		& \text{otherwise.}
\end{cases}
\end{equation}
\end{Conj}
 
The formulation of this conjecture directly in terms of the denominators 
of $A_{n,2}(x)$ is as follows.

\begin{Conj}
\label{conj-2}
The denominator $\alpha_{n,2}$ of $A_{n,2}(x)$ is given by 
\begin{equation}
\label{alphan2}
\alpha_{n,2} =
\begin{cases}
	1				& \quad \text{if $n = 1$} \\
	n! \, \lcm(1, 2, \dots, n)/6	& \quad \text{if $n = 2 \cdot 3^{m}$ for some $m \geq 1$} \\
	n! \, \lcm(1, 2, \dots, n)/2	& \quad \text{otherwise.}
\end{cases}
\end{equation}
\end{Conj}

This conjecture shows that the cancellations produced by the polynomials 
$\varphi_{m}(x)$ in~\eqref{an2-type2} have an arithmetical nature.

\begin{proof}[Proof that Conjecture~\ref{conj-2} implies Conjecture~\ref{conj-1}]
Assume that~\eqref{alphan2} holds for $n \geq 1$.  If $n = 2 \cdot 3^m$, then $\alpha_{n,2}$ contains one fewer power of $3$ than $n \, \alpha_{n-1,2}$.  If $n = 2 \cdot 3^m + 1$, then $\alpha_{n,2}$ contains one more power of $3$ than $n \, \alpha_{n-1,2}$.  If $n = p^r$ is a prime power, then $\alpha_{n,2}$ contains one more power of $p$ than $n \, \alpha_{n-1,2}$.  Otherwise each prime appears the same number of times in $\alpha_{n,2}$ and $n \, \alpha_{n-1,2}$.
\end{proof}

The first reduction is obtained by expanding the inner sum in~\eqref{an2-type2}.  Define
\begin{equation}
	G_{n}(x) = -2 i \sum_{k=0}^{\lfloor{n/2 \rfloor} - 1} (-1)^k \left[ \sum_{j=2k+1}^{n-1} \frac{1}{(n-j)(n-j+1)} \binom{j}{2k+1} \right] x^{n-2k}. 
\end{equation}

\begin{Prop}
We have
\begin{equation}
	A_{n,2}(x) = \frac{i}{n!} \left[ x \left( (x+i)^{n} - (x-i)^{n} \right) + G_{n}(x) \right].
\end{equation}
\end{Prop}

\begin{proof}
Expanding the terms $(x+i)^{n-k+1}$ and $(x-i)^{n-k+1}$ in the expression for 
$A_{n,2}(x)$ yields the sum
\begin{equation}
\sum_{j=1}^{n-1} \frac{x^{n+1-j}}{(n-j)(n-j+1)} 
\sum_{k=0}^{j} \binom{j}{k} x^{j-k} i^{k} \left( (-1)^k - 1 \right) 
\end{equation}
so only odd $k$ contribute to it. Reversing the order of summation 
gives the result.
\end{proof}

The next result compares the denominator $\alpha_{n,2}$ of $A_{n,2}(x)$ and 
the denominator of $G_{n}$, denoted by $\gamma_{n}$.

\begin{Cor}
For $n \geq 1$, the denominators $\alpha_{n,2}$ and $\gamma_{n}$ satisfy
\begin{equation}
\alpha_{n,2} = n! \, \gamma_{n}.
\end{equation}
\end{Cor}

We now rephrase Conjecture~\ref{conj-2} as the following.

\begin{Conj}
\label{conj-3}
For $n \geq 2$, 
\begin{equation}
\gamma_{n} = 
\begin{cases}
	\lcm(1, 2, \dots, n)/6		& \text{if $n = 2 \cdot 3^{m}$ for some $m \geq 1$} \\
	\lcm(1, 2, \dots, n)/2		& \text{otherwise.}
\end{cases}
\end{equation}
\end{Conj}

The next theorem establishes part of this conjecture, namely the exceptional role that the prime $p=3$ plays.  The proof employs the notation
\begin{equation}
	g_{n,k}(j) = \frac{1}{(n-j)(n-j+1)} \binom{j}{2k+1} 
\end{equation}
so that
\begin{equation}
	G_{n}(x) = -2i \sum_{k=0}^{\lfloor{ n/2 \rfloor}-1} (-1)^{k} h_{n,k}x^{n-2k}
\end{equation}
with 
\begin{equation}
	h_{n,k} \colonequal \sum_{j=2k+1}^{n-1} g_{n,k}(j) = 
\sum_{\ell=1}^{n-1-2k} \frac{1}{\ell(\ell+1)} \binom{n-\ell}{2k+1}.
\end{equation}
Therefore, for $n \geq 2$,
\begin{equation}
\label{form-gamma-n}
	\gamma_{n} = \frac{1}{2} \cdot \lcm \left\{ \text{denominator of $h_{n,k}$} : \, 0 \leq k \leq \lfloor n/2 \rfloor - 1 \right\}.
\end{equation}

Let $\nu_p(n)$ be the exponent of the highest power of $p$ dividing $n$ --- the \emph{$p$-adic valuation of $n$}.  The denominators in the terms forming the sum $h_{n,k}$ are consecutive 
integers bounded by $n$. Therefore 
\begin{equation}
	\nu_{3}(\gamma_{n}) \leq \nu_{3}(\lcm(1, 2, \dots, n)).
\label{bound-1}
\end{equation}
In fact we can establish $\nu_{3}(\gamma_{n})$ precisely.

\begin{Thm}
\label{3-adic}
The $3$-adic valuation of $\gamma_{n}$ is given by 
\[
\nu_{3}(\gamma_{n}) =
\begin{cases}
	\nu_{3}(\lcm(1, 2, \dots, n)) - 1	& \text{if $n = 2 \cdot 3^{m}$ for some $m \geq 1$} \\
	\nu_{3}(\lcm(1, 2, \dots, n))		& \text{otherwise.}
\end{cases}
\]
\end{Thm}

\begin{proof}
The analysis is divided into two cases. 

\smallskip

\noindent
{\bf Case 1}. Assume that $n = 2 \cdot 3^{m}$. We show that $\nu_{3}(\gamma_{n}) = m-1.$ \\

The bound~\eqref{bound-1} shows that $\nu_{3}(\gamma_{n}) \leq m$. 

\smallskip

\noindent
{\em Claim}: $\nu_{3}(\gamma_{n}) \neq m$. To prove this, the coefficient
\begin{equation}
h_{n,k} = \sum_{\ell=1}^{n-1-2k} \frac{1}{\ell(\ell+1)} \binom{n-\ell}{2k+1}
\end{equation}
is written as 
\begin{equation}
h_{n,k} = S_{1}(n,k) + S_{2}(n,k)
\end{equation}
where $S_{1}(n,k)$ is the sum of all the terms in $h_{n,k}$ with a 
denominator divisible by $3^{m}$ and $S_{2}(n,k)$ contains the remaining terms.
This is the highest possible power of $3$ that appears in the denominator of 
$h_{n,k}$. \\

It is now shown that the denominator of the sum $S_{1}(n,k)$ is never divisible 
by $3^{m}$.

\smallskip

\noindent
{\bf Step 1}. The sum $S_{1}(n,k)$ contains at most two 
terms.

\begin{proof}
The index $\ell$ satisfies $\ell \leq 2 \cdot 3^{m}-1-2k < 2 \cdot 3^{m}$. The 
only 
choices of $\ell$ that produce denominators divisible by $3^{m}$ are 
$\ell = 3^{m}, 3^{m}-1$ and $\ell=2 \cdot 3^{m}-1$. The term 
corresponding to this 
last choice is $\frac{1}{(2 \cdot 3^{m}-1) \cdot 2 \cdot 3^{m}} 
\binom{1}{2k+1}$, so it 
only occurs for $k=0$. In this situation, the term corresponding 
to $\ell=3^{m}$ is $1/(3^{m}+1)$ and it does not contribute to $S_{1}$.
\end{proof}

\noindent
{\bf Step 2}. If $\tfrac{1}{2}(3^{m}-1) < k \leq 3^m - 1$, then $S_{1}(n,k)$ is the empty sum. 
Therefore the denominator of $h_{n,k}$ is not divisible by $3^{m}$.

\begin{proof}
The index $\ell$ in the sum defining $h_{n,k}$ satisfies $1 \leq \ell \leq 
2 \cdot 3^{m}-1-2k$. The assumption on $k$ guarantees that 
neither $\ell = 3^{m}$
nor $\ell=3^{m}-1$ appear in this range.
\end{proof}

\noindent
{\bf Step 3}. If $k = \tfrac{1}{2}(3^{m}-1)$, then the denominator of $S_{1}(n,k)$ is not divisible by $3^{m}$.

\begin{proof}
In this case the sum $S_{1}(n,k)$ is
\[
	\frac{1}{3^m (3^m + 1)} + \frac{3^m + 1}{(3^m - 1) 3^m} = \frac{3^m + 3}{3^{2m} - 1}. \qedhere
\]
\end{proof}

\noindent
{\bf Step 4}. If $0 < k < \tfrac{1}{2}(3^{m}-1)$, then the denominator of 
$S_{1}(n,k)$ is not divisible by $3^{m}$.

\begin{proof}
The proof of this step employs a theorem of Kummer stating that $\nu_p(\binom{a}{b})$ is equal to the number of borrows involved in subtracting $b$ from $a$ in base $p$.
By Kummer's theorem, $\binom{3^m}{2 k + 1}$ and $\binom{3^m + 1}{2 k + 1}$ are divisible by $3$, so neither of the two terms in $S_{1}(n,k)$ has denominator divisible by $3^m$.
\end{proof}

\noindent
{\bf Step 5}. If $k = 0$, then the denominator of $S_{1}(n,k)$ is not divisible by $3^{m}$.

\begin{proof}
For $k=0$ we have
\begin{equation}
h_{n,0} = \sum_{\ell=1}^{n-1} \frac{n-\ell}{\ell(\ell+1)} = 
\sum_{\ell=1}^{n-1} \left( \frac{n - \ell}{\ell} - \frac{n - (\ell + 1)}
{\ell + 1} - \frac{1}{\ell + 1} \right) = n - H_{n},
\end{equation}
and the two terms in $H_n$ whose denominators are divisible by $3^m$ add up to 
\begin{equation}
\frac{1}{3^{m}} + \frac{1}{2 \cdot 3^{m}} = \frac{1}{2 \cdot 3^{m-1}}
\end{equation}
with denominator not divisible by $3^{m}$. 
\end{proof}

It follows that, for $n = 2 \cdot 3^m$, the denominator of the term $h_{n,k}$ is not divisible by 
$3^{m}$. Thus, $\nu_{3}(\gamma_{n}) \leq m-1$. 

\medskip

\noindent
{\em Claim}: $\nu_{3}(\gamma_{n}) \geq m-1$. This is established 
by checking that $3^{m-1}$
divides the denominator of $h_{n,0}$. Indeed, there are six terms in 
$h_{n,0} = n - H_{n}$ whose denominators are divisible by $3^{m-1}$, and 
their sum is 
\begin{equation}
\sum_{\ell=1}^{6} \frac{1}{\ell \cdot 3^{m-1}} = \frac{H_{6}}{3^{m-1}} = 
\frac{49}{20 \cdot 3^{m-1}}.
\end{equation}
Therefore $3^{m-1}$ divides the denominator of $h_{n,0}$. This completes Case 1.

\medskip

\noindent
{\bf Case 2}. Assume now that $n$ is not of the form $2 \cdot 3^{m}$. This 
states that the base $3$ representation of $n$ is not of 
the form $200 \cdots 00_3$.

Let $r = \lfloor{ \log_3 n \rfloor}$, so that $3^{r}$ is the 
largest power of $3$ less than or equal to $n$. We show that $\nu_{3}(\gamma_{n}) = r$
by exhibiting a value of the index $k$ so that the 
denominator of $h_{n,k}$ is divisible by $3^{r}$.

\smallskip

\noindent
{\bf Step 1}. Assume first that the base $3$ representation of $n$ begins with $1$.
Then choose $k=0$. As before, $h_{n,0} = n - H_{n}$. Observe that each 
term in the sum
\begin{equation}
	\lcm(1, 2, \dots, n) \cdot H_{n} = \sum_{\ell=1}^{n} 
\frac{\lcm(1, 2, \dots, n)}{\ell}
\end{equation}
is an integer. The condition on the base $3$ representation of $n$ guarantees 
that only one of these integers, namely the one corresponding to 
$\ell=3^{r}$, is
not divisible by $3$. Thus there is no extra cancellation of powers of $3$ in 
$H_{n}$, and as a result the denominator of $H_{n}$ is divisible by $3^{r}$. 

\smallskip

\noindent
{\bf Step 2}. Assume now that the base $3$ representation of $n$ begins with $2$.
Choose $k = \tfrac{1}{2}( 3^{r} + 3^{\nu_{3}(n)})$.
As in the discussion in Case~1, there are at most two terms in the sum
\begin{equation}
	h_{n,k} = \sum_{\ell=1}^{n-1-2k} \frac{1}{\ell(\ell+1)} \binom{n-\ell}{2k+1}
\end{equation}
with denominator divisible by $3^{r}$.
The sum of these terms is
\begin{equation}
	\frac{1}{3^{r}(3^{r}+1)} \binom{n-3^{r}}{2k+1} + \frac{1}{(3^{r}-1)3^{r}} \binom{n-3^{r}+1}{2k+1}.
\label{sum-bin2}
\end{equation}
If $n \equiv 2 \mod 3$, then Kummer's theorem shows that $3$ divides the second binomial coefficient but not the first;
otherwise, $3$ divides the first binomial coefficient but not the second.
Therefore $h_{n,k}$ has precisely one term with denominator divisible by $3^{r}$.
The argument is complete.
\end{proof}

\begin{Cor}
The $3$-adic valuation of the denominator $\alpha_{n,2}$ of $A_{n,2}(x)$ is
\[
\nu_{3}(\alpha_{n,2}) = 
\begin{cases}
	\nu_{3}(n! \, \lcm(1, 2, \dots, n)) - 1	& \text{if $n = 2 \cdot 3^{m}$ for some $m \geq 1$} \\
	\nu_{3}(n! \, \lcm(1, 2, \dots, n))	& \text{otherwise.}
\end{cases}
\]
\end{Cor}

\begin{Note}
It easily follows that the $2$-adic valuation of $\gamma_{n}$ is 
\begin{equation}
\nu_{2}(\gamma_{n}) = \nu_{2}( \lcm(1, 2, \dots, n)) -1.
\label{form-gamma-n1}
\end{equation}
Indeed, from \eqref{form-gamma-n} one has
$\nu_{2}(\gamma_{n}) \leq  \nu_{2}( \lcm(1, 2, \dots, n)) -1$. On the other 
hand, the denominator of $h_{n,0} = n - H_{n}$ is divisible by the highest 
power of $2$, i.e., by $2^{\left\lfloor \log_{2}n \right\rfloor}$, which
implies \eqref{form-gamma-n1}. 

The proof of Conjecture~\ref{conj-2} has been reduced to the identity 
\begin{equation}
	\nu_{p}(\gamma_{n}) = \nu_{p}( \lcm(1, 2, \dots, n) )
\end{equation}
for all primes $p > 3$. 
\end{Note}

The sequence $2 \cdot 3^{m}$ appearing in the previous discussion also appears 
in relation with the denominators of the harmonic numbers $H_{n}$. As before, 
write 
\begin{equation}
	H_{n} = \frac{N_{n}}{D_{n}}
\end{equation}
in reduced form. The next result considers a special case of the quotient 
$D_{n-1}/D_{n}$ of denominators of consecutive harmonic numbers. The general 
case will be described elsewhere~\cite{moll-sun3}.

\begin{Thm}
Let $n \in \mathbb{N}$. Then $D_{2 \cdot 3^{n}-1} = 3 D_{2 \cdot 3^{n}}$. 
\end{Thm}

\begin{proof}
An elementary argument shows that $\nu_{2}(D_{n}) = 
\lfloor{ \log_2 n \rfloor}$. Therefore $N_{n}$ is odd and $D_{n}$ is even. 

Observe that 
\begin{align}
	\frac{N_{2 \cdot 3^{n}}}{D_{2 \cdot 3^{n}}}
		&= \frac{N_{2 \cdot 3^{n}-1}}{D_{2 \cdot 3^{n}-1}} + \frac{1}{2 \cdot 3^{n}} \label{quot-1} \\
		&= \frac{2 \cdot 3^{n} N_{2 \cdot 3^{n}-1} + D_{2 \cdot 3^{n}-1}}{2 \cdot 3^{n} D_{2 \cdot 3^{n}-1}}. \nonumber
\end{align}
Therefore the denominator $D_{2 \cdot 3^{n}}$ is obtained from 
$2 \cdot 3^{n} D_{2 \cdot 3^{n} -1}$ by canceling the factor 
\begin{equation}
w = \gcd \left( 
2 \cdot 3^{n} \cdot N_{2 \cdot 3^{n}-1} + D_{2 \cdot 3^{n}-1}, 
2 \cdot 3^{n} \cdot D_{2 \cdot 3^{n}-1} \right).
\end{equation}
That is, 
\begin{equation}
2 \cdot 3^{n} \cdot D_{2 \cdot 3^{n}-1} = w \cdot D_{2 \cdot 3^{n}}.
\label{reduce-1}
\end{equation}

\begin{Lem}
\label{prop-2and3}
The number $w$ has the form $2^{\alpha} \cdot 3^{\beta}$, for some 
$\alpha, \beta \geq 0$. 
\end{Lem}

\begin{proof}
Any prime factor $p$ of $w$ divides 
$$ 2 \cdot 3^{n} \cdot 
\left(2 \cdot 3^{n} \cdot N_{2 \cdot 3^{n}-1} + D_{2 \cdot 3^{n}-1} \right) - 
2 \cdot 3^{n} \cdot D_{2 \cdot 3^{n}-1} = 
2^{2} \cdot 3^{2n} \cdot 
N_{2 \cdot 3^{n}-1}.$$
Then $p$ is a common divisor of 
$2 \cdot 3^{n} \cdot N_{2 \cdot 3^{n}-1}$ and
$2 \cdot 3^{n} \cdot D_{2 \cdot 3^{n}-1}$.  The harmonic 
numbers are in reduced form, so $p$ must be $2$ or $3$.
\end{proof}

The relation~\eqref{reduce-1} becomes
$2 \cdot 3^{n} D_{2 \cdot 3^{n}-1} = 2^{\alpha} \cdot 3^{\beta} 
D_{2 \cdot 3^{n}}$, and 
replacing this in~\eqref{quot-1} yields
\begin{equation}
2^{\alpha} \cdot 3^{\beta} N_{2 \cdot 3^{n}} = 
2 \cdot 3^{n} N_{2 \cdot 3^{n}-1} + D_{2 \cdot 3^{n}-1}. 
\label{quot-2}
\end{equation}
Define $t = \lfloor{ \log_2 (2 \cdot 3^{n}-1) \rfloor} > 1$ and write 
$D_{2 \cdot 3^{n}-1} = 2^{t} C_{2 \cdot 3^{n}-1}$ with $C_{2 \cdot 3^{n}-1}$
an odd integer. Then~\eqref{quot-2} becomes 
\begin{equation}
2^{\alpha-1} \cdot 3^{\beta} N_{2 \cdot 3^{n}} - 2^{t-1} C_{2 \cdot 3^{n}-1} 
= 3^{n} N_{2 \cdot 3^{n}-1}. 
\label{quot-3}
\end{equation}
A simple analysis of the parity of each term in~\eqref{quot-3} shows that 
the only possibility is $\alpha=1$. 

The relation~\eqref{reduce-1} now becomes 
\begin{equation}
3^{n} \cdot D_{2 \cdot 3^{n}-1} = 3^{\beta} \cdot D_{2 \cdot 3^{n}}.
\label{quot-4}
\end{equation}
In the computation of the denominator $D_{2 \cdot 3^{n}}$ we have the sum
\begin{equation}
1 + \frac{1}{2} + \frac{1}{3} + \cdots + \frac{1}{3^{n}} + \cdots + 
\frac{1}{2 \cdot 3^{n}-1} + \frac{1}{2 \cdot 3^{n}}
\label{sum-1}
\end{equation}
so that the maximum power of $3$ that appears in a denominator forming the 
sum~\eqref{sum-1} is $3^{n}$. 
Simply observe that $3^{n+1} > 2 \cdot 3^{n}-1$. The combination of all the 
fractions in the sum~\eqref{sum-1} with denominator $3^{n}$ is 
\begin{equation}
\frac{1}{3^{n}} + \frac{1}{2 \cdot 3^{n}} = \frac{2+1}{2 \cdot 3^{n}} = 
\frac{1}{2 \cdot 3^{n-1}}.
\end{equation}
It follows that the maximum power of $3$ in~\eqref{sum-1} is at most $3^{n-1}$.

The terms in~\eqref{sum-1} that contain exactly $3^{n-1}$ in 
the denominator are 
\begin{equation}
	\frac{1}{3^{n-1}}, \, \frac{1}{2 \cdot 3^{n-1}}, \, \frac{1}{4 \cdot 3^{n-1}}, \, \frac{1}{5 \cdot 3^{n-1}},
\end{equation}
and these combine with the two terms with denominator exactly divisible by $3^{n}$ 
to produce
\begin{equation}
\left( 1 + \frac{1}{2} + \frac{1}{4} + \frac{1}{5} \right) \cdot 
\frac{1}{3^{n-1}} + \frac{1}{2 \cdot 3^{n-1}} = 
\frac{49}{20 \cdot 3^{n-1}}.
\end{equation}
The rest of the terms in~\eqref{sum-1} have at most a power of $3^{n-2}$ in 
the denominator. The total sum can be written as 
\begin{equation}
\frac{49}{20 \cdot 3^{n-1}} + 
\frac{x_{n}}{y_{n} \cdot 3^{n-2}} = 
\frac{49 y_{n} + 60 x_{n}}{20 y_{n} \cdot 3^{n-1}}
\end{equation}
and no cancellation occurs. Therefore
$3^{n-1}$ is the $3$-adic valuation of $D_{2 \cdot 3^{n}}$. Write 
$D_{2 \cdot 3^{n}} = 3^{n-1} \cdot E_{2 \cdot 3^{n}}$, where 
$E_{2 \cdot 3^{n}}$ is not divisible by $3$. 

Now consider the denominator $D_{2 \cdot 3^{n}-1}$.  Observe that 
\begin{equation}
1 + \frac{1}{2} + \cdots + \frac{1}{2 \cdot 3^{n}-1} = 
\frac{1}{3^{n}} + \frac{x_{n}}{y_{n} \cdot 3^{n-1}}
= \frac{y_{n} + 3 x_{n}}{y_{n} \cdot 3^{n}},
\end{equation}
with $y_{n}$ not divisible by $3$. Therefore 
$3^{n}$ is the $3$-adic 
valuation of $D_{2 \cdot 3^{n}-1}$. Write 
$D_{2 \cdot 3^{n}-1} = 3^{n} \cdot E_{2 \cdot 3^{n}-1}$ where 
$E_{2 \cdot 3^{n}-1}$ is not divisible by $3$. 

The relation~\eqref{quot-4} now reads
$3^{2n} E_{2 \cdot 3^{n}-1} = 3^{\beta +n-1} E_{2 \cdot 3^{n}}$
and this gives $\beta = n+1$. Replacing in~\eqref{quot-4} produces
$D_{2 \cdot 3^{n}-1} = 3 D_{2 \cdot 3^{n}}$, as claimed. 
\end{proof}

\section{The iterated integral of $\ln(1+x^{3})$} \label{sec-case3}
\setcounter{equation}{0}

In this final section we consider the iterated integral of $\ln(1+x^{3})$.  The first value is
\begin{multline*}
	f_{1}(x) =
	\frac{1}{6} \left(\sqrt{3} \pi - 18 x\right) \\
	- \sqrt{3} \arctan\left(\frac{1 - 2 x}{\sqrt{3}}\right)
	+ (x + 1) \ln(x + 1)
	+ \frac{1}{2} (2 x-1) \ln(x^2 - x + 1).
\end{multline*}
This and additional values suggest the ansatz
\begin{equation}
	f_{n}(x) = A_{n,3}(x) + B_{n,3}(x) u + C_{n,3}(x) v + D_{n,3}(x) w,
\label{form3-ans}
\end{equation}
where
\begin{align*}
	u &= \sqrt{3} \arctan \left( \frac{1 - 2x}{\sqrt{3}} \right) \\
	v &= \ln(x + 1) \\
	w &= \ln(x^2-x+1)
\end{align*}
and where $A_{n,3}$ is a polynomial in $\mathbb{Q}[\sqrt{3} \pi, x]$ and $B_{n,3}$, $C_{n,3}$, and $D_{n,3}$ are polynomials in $\mathbb{Q}[x]$.

The method of roots described in Section~\ref{sec-roots} shows that $f_n(x)$ can be expressed in terms of $\ln(x + 1)$, $\ln(x + \omega)$, and $\ln(x + \bar{\omega})$,
where $\omega = e^{2 \pi i/3} = \tfrac{1}{2} (-1 + i \sqrt{3})$ satisfies $\omega^{3} = 1$.
Using the relation
\begin{equation}
	\ln(-2 i (x + \omega)) = \frac{1}{2} \ln(x^2 - x + 1) + i \arctan\left(\frac{1 - 2 x}{\sqrt{3}}\right) + \ln(2),
\end{equation}
we can convert between \eqref{form3-ans} and expressions in terms of $\ln(x + \omega)$ and $\ln(x + \bar{\omega})$.

As was the case for the iterated integrals of $\ln(1+x)$ and $\ln(1+x^2)$, it is easy to conjecture closed forms for all but one of these polynomials.

\begin{Thm}
Define 
\[
\chi_{3}(k) =
\begin{cases}
	0	& \text{if $k \equiv 0 \mod 3$} \\
	1	& \text{if $k \equiv 1 \mod 3$} \\
	-1	& \text{if $k \equiv 2 \mod 3$}
\end{cases}
\]
and 
\[
\lambda(k) =
\begin{cases}
	1	& \text{if $k \equiv 0 \mod 3$} \\
	0	& \text{if $k \nequiv 0 \mod 3$.}
\end{cases}
\]
Then 
\begin{align*}
	B_{n,3}(x) & = - \frac{1}{n!} \sum_{k=0}^n \chi_{3}(n-k) \binom{n}{k}x^{k} \\
	C_{n,3}(x) & = \frac{1}{n!}(x+1)^{n} = \frac{1}{n!} \sum_{k=0}^{n} \binom{n}{k}x^{k} \\
	D_{n,3}(x) & = \frac{1}{2n!} \sum_{k=0}^{n} (3 \lambda(n-k) - 1) \binom{n}{k}x^{k}.
\end{align*}
\end{Thm}

\begin{proof}
The method of roots developed in Section~\ref{sec-roots} shows that the iterated
integral can be expressed in the form~\eqref{form3-ans}. The polynomials 
$A_{n,3}, B_{n,3}, C_{n,3}, D_{n,3}$ will be linear combinations of the powers $(x+1)^{n}$, $(x+\omega)^{n}$, and 
$(x+\bar{\omega})^{n}$.
Comparing initial values, it is found that
\begin{align*}
	B_{n,3}(x) & = \frac{i}{\sqrt{3} n!}\big((x+\omega)^n - (x+\bar{\omega})^n\big) \\
	C_{n,3}(x) & = \frac{1}{n!}(x+1)^n \\
	D_{n,3}(x) & = \frac{1}{2 n!}\big((x+\omega)^n + (x+\bar{\omega})^n\big).
\end{align*}
Note that the above expressions can also be automatically found as
solutions of the fourth-order recurrence that \texttt{HolonomicFunctions}
derives in this case:
\begin{multline}
\label{rec-3}
(n-2)(n-1)n^{2} F_{n} \\
	= (n-2)(n-1)(4n-3)x F_{n-1} - 3(n-2)(2n-3)x^{2} F_{n-2} \\
	+ \left[ (4n-9)x^{3} + n \right]F_{n-3} - x(x^{3}+1)F_{n-4}.
\end{multline}

The above closed forms for $B_{n,3}$ and $D_{n,3}$ can be used to derive
explicit expressions for their coefficients:
\[
	B_{n,3}(x) = \frac{i}{\sqrt{3} n!} \sum_{k=0}^n\binom{n}{k}x^k\left(\omega^{n-k}-\bar{\omega}^{n-k}\right).
\]
The value of the last parenthesis can be found by case distinction
using the fact that $\omega^3=\bar{\omega}^3=1$:
\begin{eqnarray*}
	n-k\equiv 0\mod 3: & 1-1=0,\\
	n-k\equiv 1\mod 3: & \omega-\bar{\omega} = i\sqrt3,\\
	n-k\equiv 2\mod 3: & \omega^2-\bar{\omega}^2 = -i\sqrt3.
\end{eqnarray*}
It follows that 
\[ 
	B_{n,3}(x) = -\frac{1}{n!}\sum_{k=0}^{n}\chi_3(n-k)\binom{n}{k}x^k.
\]
The similar computation for $D_{n,3}(x)$ is left to the reader.
\end{proof}

\begin{Note}
In order to obtain these functions from a purely 
symbolic approach, consider a brute force evaluation of $f_n(x)$ by using \emph{Mathematica} to evaluate~\eqref{expan-1}. The results are 
expressed in terms of the functions
\begin{equation}
	h_{1}(x) = {_{2}F_{1}}\Big(\begin{array}{c} 1/3, \, 1 \\ 4/3 \end{array} ; - x^3\Big)
 \quad \text{and} \quad
	h_{2}(x) = {_{2}F_{1}}\Big(\begin{array}{c} 2/3, \, 1 \\ 5/3 \end{array} ; - x^3\Big),
\label{hyper-cube}
\end{equation}
where 
\begin{equation}
	{_{2}F_{1}}\Big(\begin{array}{c} a, \, b \\ c \end{array} ; z\Big) = \sum_{k=0}^{\infty} \frac{(a)_{k} \, (b)_{k}}{(c)_{k} \, k!} z^{k}
\end{equation}
is the classical hypergeometric series. The first few values are
\begin{align*}
	f_{1}(x) & = -3x + x \, \ln(1+x^{3}) + 3x h_{1}(x) \\
	f_{2}(x) & = - \frac{9x^{2}}{4} + \frac{1}{2} x^{2} \ln(1 + x^{3}) + 3x^{2} h_{1}(x) - \frac{3x^{2}}{4} h_{2}(x) \\
	f_{3}(x) & = - \frac{11x^{3}}{12} + \frac{1}{6} (x^{3} + 1) \ln(1 + x^{3}) + \frac{3x^{3}}{2} h_{1}(x) - \frac{3x^{3}}{4} h_{2}(x).
\end{align*}
The hypergeometric series $h_1(x)$ and $h_2(x)$ can be expressed as 
\begin{align*}
	h_{1}(x) & = \frac{1}{3 x} \left(\ln(1 + x) + \omega \ln(1 + \bar{\omega} x) + \bar{\omega} \ln(1 + \omega x)\right) \\
	h_{2}(x) & = \frac{2}{3 x^2} \left(-\ln(1 + x) - \omega \ln(1 + \omega x) - \bar{\omega} \ln(1 + \bar{\omega} x)\right),
\end{align*}
with $\omega = e^{2 \pi i/3}$ as before.
These expressions can be transformed into the functions in~\eqref{form3-ans}.
\end{Note}

As in previous sections, the closed-form expression for the {\em pure} polynomial part $A_{n,3}(x)$ is more elaborate.
By writing
\[
	f_n(x) = \tilde{A}_{n,3}(x) + \tilde{B}_{n,3}(x) \ln(x+1) + \tilde{C}_{n,3}(x) \ln(x + \omega) + \tilde{D}_{n,3}(x) \ln(x + \bar{\omega}),
\]
we obtain a polynomial $\tilde{A}_{n,3}(x)$ with rational coefficients.
The first values are given by
\begin{equation}
	\tilde{A}_{0,3}(x) = 0, \quad \tilde{A}_{1,3}(x) = -3x, \quad \tilde{A}_{2,3}(x) = - \tfrac{9}{4}x^{2}, \quad \tilde{A}_{3,3}(x) = - \tfrac{11}{12}x^{3}.
\end{equation}

Schneider's \emph{Mathematica} package \texttt{Sigma} can be used again to obtain, from the recurrence~\eqref{rec-3} and the initial conditions, the expression
\begin{equation}
	\tilde{A}_{n,3}(x) = -\frac{1}{n!} \sum_{k=1}^{n} \frac{x^{k}}{k} \left[ (x+1)^{n-k} + (x + \omega)^{n-k} + (x + \bar{\omega})^{n-k} \right].
\end{equation}

\subsection{Arithmetical properties of $\tilde{A}_{n,3}$}

Define $\alpha_{n,3}$ to be the denominator of $\tilde{A}_{n,3}$ and 
$\beta_{n,3} = 
\alpha_{n,3}/(n \alpha_{n-1,3})$.

\begin{Conj}
The sequence $\beta_{n,3}$ is given by 
\begin{equation}
\label{betan3}
\beta_{n,3} =
\begin{cases}
	p		& \text{if $n = p^{m} \neq 3$ for some prime $p$ and $m \in \mathbb{N}$} \\
	\frac{1}{11}	& \text{if $n = 3 \cdot 11^{m}$ for some $m \in \mathbb{N}$} \\
	11		& \text{if $n = 3 \cdot 11^{m} + 1$ for some $m \in \mathbb{N}$} \\
	1		& \text{otherwise.}
\end{cases}
\end{equation}
\end{Conj}

Observe that this expression for $\beta_{n,3}$ does not have the 
exceptional case where $3 \cdot 11^{m} + 1$ is a prime power that 
appears in $\beta_{n,2}$ given in~\eqref{betan2}. This is ruled out by the following.

\begin{Lem}
Let $m \in \mathbb{N}$. Then $3 \cdot 11^{m}+1$ is not a prime power.
\end{Lem}

\begin{proof}
The number $3 \cdot 11^{m} + 1$ is even, so only the prime $2$ needs to be checked.  We have $3 \cdot 11^{m} + 1 \equiv 3^{m+1} + 1 \nequiv 0 \mod 8$ since the powers of $3$ are $1$ or $3$ modulo $8$.  Therefore $3 \cdot 11^{m} + 1$ (since it is larger than $4$) is not a power of $2$.
\end{proof}

\noindent
{\bf Acknowledgements}.
The authors wish to thank Xinyu Sun for discussions 
on the paper, specially on Section 5.
The authors also wish to thank the referee for a detailed reading of the paper and many corrections.
The 
work of the third author was 
partially supported by NSF DMS 0070567.
The work of the second author 
was partially supported by the same grant as a postdoctoral fellow at 
Tulane University and by the Austrian Science Fund (FWF):P20162-N18.
The work of the last author was partially supported by 
Tulane VIGRE Grant 0239996.


\begin{thebibliography}{10}



\bibitem{apelblat}
A.~Apelblat.
\newblock {\em Tables of Integrals and Series}.
\newblock Verlag Harry Deutsch, Thun; Frankfurt am Main, 1996.



\bibitem{boesmo}
G.~Boros, O.~Espinosa, and V.~Moll.
\newblock On some families of integrals solvable in terms of polygamma and
 negapolygamma functions.
\newblock {\em Integrals Transforms and Special Functions}, 14:187--203,
  2003.



\bibitem{bronstein2}
M.~Bronstein.
\newblock {\em Symbolic Integration I. Transcendental functions}, volume~1 of
  {\em Algorithms and Computation in Mathematics}.
\newblock Springer-Verlag, 1997.



\bibitem{gr}
I.~S. Gradshteyn and I.~M. Ryzhik.
\newblock {\em Table of Integrals, Series, and Products}.
\newblock Edited by A. Jeffrey and D. Zwillinger. Academic Press, New York, 7th
  edition, 2007.



\bibitem{graham1}
R.~Graham, D.~Knuth, and O.~Patashnik.
\newblock {\em Concrete Mathematics}.
\newblock Addison Wesley, Boston, 2nd edition, 1994.



\bibitem{Koutschan09}
C.~Koutschan.
\newblock {\em Advanced Applications of the Holonomic Systems Approach}.
\newblock PhD thesis, RISC, Johannes Kepler University, Linz, Austria, 2009.



\bibitem{moll-gr18}
C.~Koutschan and V.~Moll.
\newblock The integrals in Gradshteyn and Ryzhik. Part 18: Some
  automatic proofs.
\newblock {\em Scientia}, 20:93--111, 2011.



\bibitem{liouville-1834c}
J.~Liouville.
\newblock Sur les transcendentes elliptiques de premi\'{e}re et de seconde
  esp\'{e}ce, consid\'{e}r\'{e}es comme fonctions de leur amplitude.
\newblock {\em Journ. Ec. Polyt.}, 14, (23. cahier):37--83, 1834.



\bibitem{lutzen}
J.~L\"{u}tzen.
\newblock {\em Joseph Liouville $1809-1882$. Master of Pure and Applied
 Mathematics}, volume~15 of {\em Studies in the History of Mathematics
 and Physical Sciences}.
\newblock Springer-Verlag, New York, 1990.



\bibitem{manna-moll2}
D.~Manna and V.~Moll.
\newblock Rational Landen transformations on $\mathbb{R}$.
\newblock {\em Math. Comp.}, 76:2023--2043, 2007.



\bibitem{manna-moll1}
D.~Manna and V.~Moll.
\newblock A simple example of a new class of Landen transformations.
\newblock {\em Amer. Math. Monthly}, 114:232--241, 2007.



\bibitem{manna-moll3}
D.~Manna and V.~Moll.
\newblock Landen Survey.
\newblock {\em MSRI Publications: Probability, Geometry and Integrable Systems.
  In honor of Henry McKean $75$th birthday}, 55:201--233, 2008.



\bibitem{iter-loga1}
L.~Medina, V.~Moll, and E.~Rowland.
\newblock Iterated primitives of logarithmic powers.
\newblock {\em International Journal of Number Theory}, To appear, 2011.



\bibitem{moll-notices}
V.~Moll.
\newblock The evaluation of integrals: a personal story.
\newblock {\em Notices of the AMS}, 49:311--317, 2002.



\bibitem{moll-seized}
V.~Moll.
\newblock Seized opportunities.
\newblock {\em Notices of the AMS}, pages 476--484, 2010.



\bibitem{aequalsb}
M.~Petkov\v{s}ek, H.~Wilf, and D.~Zeilberger.
\newblock {\em A=B}.
\newblock A. K. Peters, Ltd., 1st edition, 1996.



\bibitem{Petkovsek92}
M.~Petkov\v{s}ek.
\newblock Hypergeometric solutions of linear recurrences with polynomial
 coefficients.
\newblock {\em Journal of Symbolic Computation}, 14(2/3):243--264, 1992.



\bibitem{prudnikov1}
A.~P. Prudnikov, Yu.~A. Brychkov, and O.~I. Marichev.
\newblock {\em Integrals and Series}.
\newblock Gordon and Breach Science Publishers, 1992.



\bibitem{risch1}
R.~H. Risch.
\newblock The problem of integration in finite terms.
\newblock {\em Trans. Amer. Math. Soc.}, 139:167--189, 1969.



\bibitem{risch2}
R.~H. Risch.
\newblock The solution of the problem of integration in finite terms.
\newblock {\em Bull. Amer. Math. Soc.}, 76:605--608, 1970.



\bibitem{ritt-1948}
J.~F. Ritt.
\newblock {\em Integration in finite terms. Liouville's theory of elementary
  functions}.
\newblock New York, 1948.



\bibitem{Schneider07}
C.~Schneider.
\newblock Symbolic summation assists Combinatorics.
\newblock {\em S\'e\-mi\-naire Lotharingien de Combinatoire}, 56:1--36, April
  2007.
\newblock Article B56b.



\bibitem{moll-sun3}
X.~Sun and V.~Moll.
\newblock Denominators of harmonic numbers.
\newblock {\em In preparation}, 2011.



\end{thebibliography}
\end{document}